\documentclass[11pt,reqno]{amsart}
\usepackage{amsfonts,amssymb,amsmath,amsopn,amsthm,graphicx}
\usepackage{amsxtra,mathrsfs}
\usepackage[mathscr]{eucal}
\usepackage{colordvi}
\usepackage[usenames,dvipsnames]{color}
\usepackage{mathtools}
\usepackage[colorlinks=true]{hyperref}
\usepackage{setspace}
\usepackage{enumitem}

\numberwithin{equation}{section}

\newtheorem{thm}{Theorem}[section]

\newtheorem{lem}[thm]{Lemma}
\newtheorem{prop}[thm]{Proposition}

\newtheorem{definition}[thm]{Definition}

\newtheorem{rem}[thm]{Remark}
\newtheorem{rems}[thm]{Remarks}

\theoremstyle{definition}
\theoremstyle{remark}

\newcommand{\s}{\mathbb{S}}

\newcommand{\N}{\mathbb{N}}
\newcommand{\R}{\mathbb{R}}

\newcommand{\loc}{{\rm loc}}

\newcommand{\jap}[1]{\langle{#1}\rangle}

\newcommand\footnoteref[1]{\protected@xdef\@thefnmark{\ref{#1}}\@footnotemark}

\newcommand{\dx}{\,\mathrm{d}x}

\newcommand{\ds}{\,\mathrm{d}s}
\newcommand{\dr}{\,\mathrm{d}r}

\def\ga{\alpha}            \def\gg{\gamma}

            \def\gl{\lambda}

\def\Gg{\Gamma}     \def\Gd{\Delta}

\def\Gw{\Omega}              

\def\ga{\alpha}            \def\gg{\gamma}

            \def\gl{\lambda}

\def\Gg{\Gamma}     \def\Gd{\Delta}

\def\Gw{\Omega}              

\title[weak perturbation of critical  quasilinear operators]{On the behavior of the ground state energy under weak perturbation of critical  quasilinear operators in $\R^N$}

\author {Ujjal Das}

\address {Ujjal Das, BCAM -- Basque Center for Applied Mathematics,
48009 Bilbao, Spain}

\email {udas@bcamath.org, getujjaldas@gmail.com}

\author {Hynek Kova\v{r}\'{\i}k}

\address{Hynek Kova\v{r}\'{\i}k, DICATAM, Sezione di Matematica, Universit\`a degli studi di Brescia, Italy}
\email {hynek.kovarik@unibs.it}

\author{Yehuda Pinchover}

\address{Yehuda Pinchover,
Department of Mathematics, Technion - Israel Institute of
Technology,   Haifa, Israel}

\email{pincho@techunix.technion.ac.il}

\begin{document}

\begin{abstract}

\medskip
		We consider a critical quasilinear operator $-\Delta_p u +V|u|^{p-2}u$ in $\mathbb{R}^N$ perturbed by a weakly coupled potential. For $N>p$ we find the leading asymptotic 
of the lowest eigenvalue of such an operator in the weak coupling limit separately for $N>p^2 $ and $N\leq p^2$.\\[1mm] 

\end{abstract}
\maketitle
\section{\bf Introduction and main results }\label{sec_s_u}

\subsection{The set up} It is a well-known fact that  $-\Gd_p$, $1<p<\infty$, the celebrated $p$-Laplace operator,  is subcritical in $\R^N$ if and only if $p<N$. Hence, 
there exist potentials $V\not\geq 0$ such that the functional 
\begin{equation} 
Q_0[u] := \int_{\R^N} |\nabla u|^p \dx + \int_{\R^N} V |u|^p\dx  \qquad u\in W^{1,p}(\R^N),
\end{equation}  
is {\em critical} \cite[Proposition 4.4]{pt}. Being critical means that the associated equation $-\Gd_p u +V|u|^{p-2}u=0$ in $\R^N$ admits a unique (up to a multiplicative constant) positive supersolution $\phi_0\in L^p_{\loc}(\R^N)$ which is in fact positive solution $-\Gd_p u +V|u|^{p-2}u=0$ in $\R^N$. Such a solution is called an {\em Agmon ground state}.  
Moreover, $\phi_0$ is a positive solution of minimal growth at infinity (see Definition \ref{min_gr}). Accordingly, such potentials are called critical. In the sequel, the (Agmon) ground state of $Q_0$ will be normalized so that 
\begin{equation} \label{phi0-norm}
\phi_0(0)=1.
\end{equation}
The corresponding quasilinear operator will be denoted by $-\Gd_p +V$.

We consider the energy functionals
\begin{equation} 
Q_{\alpha W}[u] := \int_{\R^N} |\nabla u|^p \dx + \int_{\R^N} V |u|^p\dx -\alpha \int_{\R^N} W |u|^p\dx, \qquad u\in W^{1,p}(\R^N),
\end{equation}  
where $V$ is a real valued critical potential, $W\in C_c(\R^N)$ and $\alpha>0$ is a coupling constant.  The associated variational problem for $Q_{\alpha W}$
then reads
\begin{equation} \label{lambda-def}
\lambda(\alpha) = \inf_{0\neq u\in W^{1,p}(\R^N)} \frac{Q_{\alpha W}[u]}{\|u\|_p^p}\, .
\end{equation}

We will assume that
\begin{equation} \label{condition}
\int_{\R^N} W \phi_0^p\dx >0. 
\end{equation}
Then, by \cite[Proposition 4.5]{pt} for any $\alpha>0$ we have $\lambda(\alpha)<0$. It can be  easily verified that $\gl: [0,\infty) \to \R$ is a continuous concave function of $\alpha$. 
Our aim is to study the asymptotic behavior of the ground state energy $\lambda(\alpha)$ as $\alpha\searrow 0$. 

The asymptotic behavior of $\lambda(\alpha)$ for small $\alpha$ was extensively studied in the linear case $p=2$. In low dimensions, for $N=1,2$, we know that $V=0$ is critical and  equation \eqref{lambda-def}
then  defines  the ground state energy of the Schr\"odinger operator $-\Delta - \alpha W$. In particular, it turns out, see  \cite{bgs,kl77,si}, that for sufficiently fast decaying $W$ we have 
\begin{equation} \label{eq:p=2}
\sqrt{-\lambda(\alpha)} \, = \, \frac \alpha2\int_{\R} W \dx - C_W\, \alpha^2 + o(\alpha^2), \quad \alpha\to 0, \qquad N=1, \ p=2, 
\end{equation}
with an explicit constant $C_W$ depending on $W$, see \cite{si}.  The proof of  \eqref{eq:p=2} uses the Birman-Schwinger principle and the explicit knowledge of the unperturbed resolvent. When suitably modified, 
this method 
can be  applied also to Schr\"odinger operators with long-range potentials \cite{bgs,kl79},  to higher order and fractional Schr\"odinger operators \cite{az1,az2,ha},  linear operators with degenerate zero eigenvalues \cite{ba}, and even to certain operators with complex-valued potentials, see e.g.~\cite{fk}.
Analogous problem in dimensions $N\geq 3$ and with $V\neq 0$ was treated in \cite{KS}.
  
  \smallskip
  
Considerably less is known about the non-linear case when $p\neq 2$.  Here the operator-theoretic method mentioned above is not available and a different approach is needed. 
In \cite{efk} the problem was studied for $N\leq p$ and critical potential $V=0$. It was shown there, with purely variational methods, that 
\begin{equation} \label{limit-efk}
\lim_{\alpha\to 0+}\, \alpha^{- \frac{p}{p-N}}\ \lambda(\alpha) \, = \, C_{N,p}
\left( \int_{\R^N} W\, \dx\right)^{\frac{p}{p-N}} \,, \qquad N<p,
\end{equation}
where $C_{N,p}$ is explicitly related to the best constant $S$ in the Sobolev  inequality 
$$
\| u \|_\infty^p \leq\, S\,  \| \nabla u \|_p^{N}\, 
\| u \|_p^{p- N} \qquad \forall u \in W^{1,p}(\R^N).$$
The border-line case $p=N$, in which $\lambda(\alpha)$ is exponentially small was also studied in \cite{efk}. Let us mention that similar variational approach was used also for certain linear 
operators, \cite{bcez, fmv, hhrv}. 

  \smallskip

In this paper we show, using a combination of variational and PDE techniques, how the main contribution to the asymptotic of $\lambda(\alpha)$ depends on $\alpha$ and $W$ in the case $N > p$ and $V\neq 0$. Similarly as in \cite{efk} the asymptotic order depends 
 on the relation between $N$ and $p$.  Our main results, when put together with those of \cite{efk} are summarized in Table \ref{table}. 
 For a more precise formulation  see Theorems \ref{thm-linear} and \ref{thm-superlinear} below.

\begin{figure}[!h]
\begin{tabular}{|c|c|c|}
\hline
\rule[-0.2cm]{0mm}{0.6cm}
Dimension &\ \ \ Leading order  of $\lambda(\alpha)$ & Critical potential \\
\hline\hline
\rule[-0.4cm]{0mm}{1cm}
$N<p$ &   $\qquad\displaystyle\alpha^{\frac{p}{p-N}}$ & $V=0$\\
\hline
\rule[-0.4cm]{0mm}{1cm}
$N=p$ & $\exp\Big[ \big(-c/\alpha\big)^{\frac{1}{N-1}}\Big] $ & $V=0$\\
\hline
\rule[-0.4cm]{0mm}{1cm}
$p<N<p^2$ & $\alpha^{\frac{p(p-1)}{N-p}}$ & $V\in C_c(\R^N)$\\
\hline
\rule[-0.4cm]{0mm}{1cm}
$N=p^2$ & $\displaystyle\frac{\alpha}{|\log \alpha|}$ & $V\in C_c(\R^N)$\\
\hline
\rule[-0.4cm]{0mm}{1cm}
$N>p^2$ & $\alpha $ & $V\in L^1(\R^N)\cap L^{\infty}(\R^N)$ satisfying \eqref{eq_Fuchs}\\
\hline
\end{tabular}
\caption{Asymptotic order of $\lambda(\alpha)$. The results for $N\leq p$ are due to  \cite{efk}. }
\label{table}
\end{figure}

\subsection{Main results} We will present our results separately for $N>p^2$ and $N\leq p^2$.

\begin{thm}\label{thm-linear}
Suppose that $V\in L^1(\R^N)\cap L^{\infty}(\R^N)$ 
is a critical potential for the $p$-Laplacian in $\R^N$ satisfying \eqref{eq_Fuchs}.
Let $\phi_0$ be the corresponding ground state and let $W\in C_c(\R^N)$ satisfies \eqref{condition}. If $N>p^2$, then  
\begin{equation} \label{pos_crit}
\lim_{\alpha\to 0_+} \alpha^{-1}\lambda(\alpha) = -\frac{\int_{\R^N} W  \phi_0^p\dx}{\|\phi_0\|_p^p}\, .
\end{equation}
\end{thm}

To state our results in the case $N\leq p^2$ we need  some notation. For a positive functions $f$  we write 
$$
|\lambda(\alpha)| \asymp f(\alpha) \qquad \alpha\to 0_+\\[4pt]
$$
if there exist positive constants $K_1$ and $K_2$, {\bf independent of} $W$, such that\\[2pt] 
\begin{equation} \label{k12}
K_1 \leq \liminf_{\alpha\to 0_+} \frac{|\lambda(\alpha)|}{f(\alpha)} \leq \limsup_{\alpha\to 0_+} \frac{|\lambda(\alpha)|}{f(\alpha)}  \leq K_2\,.
\end{equation}
We then have

\begin{thm}\label{thm-superlinear}
Let $V, W\in  C_c(\R^N)$. Suppose that $V$ critical for the $p$-Laplacian in $\R^N$, and that $W\in C_c(\R^N)$ satisfies \eqref{condition}. 
\begin{enumerate}  
\item If $p<N<p^2$, then 
\begin{equation} \label{N<p^2}
|\lambda(\alpha)|\ \asymp\   \alpha^{\frac{p(p-1)}{N-p}}\  \left(\int_{\R^N} W \phi_0^p \dx\right)^{{\frac{p(p-1)}{N-p}}} \qquad \alpha\to 0_+. 
\end{equation}

\item  If $N=p^2$, then 
\begin{equation} \label{N=p^2}
|\lambda(\alpha)|\ \asymp\  \frac{\alpha}{|\log\alpha|}\, \int_{\R^N} W \phi_0^p \dx \qquad\qquad \quad\ \alpha\to 0_+. 
\end{equation}
\end{enumerate}
\end{thm}

The proof of Theorem \ref{thm-linear} is given in Section \ref{sec_minimizer}. In Section \ref{sec_lsc} we prove the upper and the lower
bounds needed for the proof of Theorem \ref{thm-superlinear}. 

\begin{rems}
{\em Some comments concerning the above theorems are in order.
\begin{enumerate}
 \item The infimum in \eqref{lambda-def} is attained as soon as $\alpha>0$ and $W$ satisfies \eqref{condition}, see  Lemma \ref{lem_mimimizer}.
To prove the lower bounds on $\lambda(\alpha)$ in  \eqref{N<p^2} and \eqref{N=p^2} we first obtain an order-sharp estimate on the blow-up of the $L^p-$norm of the associated minimizer $\phi_\alpha$. 
This is achieved by an iterated application of the comparison principle, see Propositions  \ref{prop:LB1} and  \ref{prop:LB2}.

\item The upper bounds on $\lambda(\alpha)$ are obtained by a suitable choice of a family of test functions.

\item The constants $K_1$ and $K_2$ relative to equations \eqref{N<p^2} and \eqref{N=p^2}, cf.~\eqref{k12}, depend on $V$ but not on $W$. 
Confronting the right hand sides of \eqref{N<p^2} and \eqref{N=p^2} with those of  \eqref{eq:p=2} and \eqref{limit-efk} it is important to notice 
that for $N\leq p$ and $V=0$ we have $\phi_0=1$.

\item If $\int_{\R^N} W \phi_0^p \dx  <0$, then by the criticality theory (see \cite[Proposition 4.5]{pt}), we have $\lambda(\alpha)=0$ for $\alpha>0$ small enough,  while $\lambda(\alpha)<0$ for any $\alpha>0$ even if $\int_{\R^N} W \phi_0^p \dx =0$.  So far, the asymptotic behavior of $\lambda(\alpha)$ in the latter case is known only in the linear case $p=2$, \cite{si}.

\item For $p=2$ our results agree with those obtained in \cite{KS} for Schr\"odinger operators.

\item It is natural to conjecture that the results of Theorems \ref{thm-linear} and \ref{thm-superlinear}  hold  even without assuming $W\in C_c(\R^N)$.
Indeed, the upper bounds in the above theorems hold under much weaker conditions on $W$, see equation \eqref{upperb} and Propositions \ref{prop: upperb-1} and \ref{prop: upperb-2}. 
The condition $W\in C_c(\R^N)$ can be removed even in the lower bound as long as $p\geq 2$, cf.~Proposition \ref{prop:LB-new} in  Appendix.
However, if $p<2$, then the hypothesis that $W$ is compactly supported is fundamental for our approach in the proofs of the lower bounds. 
\end{enumerate} }
\end{rems}

\subsection{Preliminaries and notation}
The following two-sided estimate of the energy functional $Q_0$ will be important for our analysis. It states that if $\phi_0$ is as above, then for every $0\leq u \in W^{1,p}_{\rm loc}(\R^N)$,  we have 
\begin{equation} \label{two-sided-pr}
Q_0[u] \, \asymp\, \int_{\R^N} \phi_0^2\, |\nabla v|^2 \big( v |\nabla \phi_0|+\phi_0 |\nabla v|\big)^{p-2}\dx, \qquad v = \frac{u}{\phi_0}\, ,
\end{equation}
where the equivalent constants depend only on $p$ and $N$, and provided that the right-hand side of \eqref{two-sided-pr} is finite. The functional in the right hand side of the above equivalence is called the {\it simplified energy functional}. For the proof of \eqref{two-sided-pr} we refer to  \cite[Lemma 2.2]{ptt} and \cite[Lemma 3.4]{pr}.
We also recall the Sobolev inequality with critical exponent,
\begin{equation} \label{sobolev}
\|u\|_{p^*} \, \leq \, C(p,N)\, \|\nabla u\|_p \qquad  u\in W^{1,p}(\R^N)\, ,
\end{equation}
where
\begin{equation} 
p^* = \frac{Np}{N-p}\, .
\end{equation}
We denote by $B_R(0)$ the open ball of radius centered in zero and by $B_R^c(0)=\R^N\setminus \overline{B_R(0)}$.

We also recall some of the notions from the quasilinear criticality theory that will be used in this article. 
\begin{definition}[Positive solution of minimal growth at infinity] \label{min_gr}
		{\em Let $K_0$ be a compact set in $\R^N$ such that $\R^N\setminus K_0$ is connected and $\omega \in L^{q}_{\loc}(\R^N)$, where $q>\max\{N/p,1\}$.  A positive solution $u$
		of the equation $[-\Delta_p+\omega](w)=0$ in $\R^N\setminus K_0$ is said to be a
		{\it  positive solution of minimal growth in a neighborhood of
			infinity in} $\R^N$ if
		for any compact set $K$ in $\R^N$, with a smooth boundary, such
		that $\R^N\setminus K$ is connected and  $K_0 \Subset \mathrm{int}(K)$, and any positive supersolution
		$v\in C((\R^N\setminus K)\cup
		\partial K)$ of the equation $[-\Delta_p +\omega](w)=0$ in $\R^N\setminus K$,
		the inequality $u\le v$ on $\partial K$ implies that $u\le v$ in
		$\R^N\setminus K$.

A positive solution $u$ of minimal growth at infinity with respect to $K_0=\emptyset$ is called a {\em global minimal solution}. }
\end{definition}

It turns out that $-\Delta_p +\omega$ admits a global minimal solution in $\R^N$ if and only if $-\Delta_p +\omega$ is critical in $\R^N$ (\cite[Theorem~5.9]{Yehuda_Georgios}). Hence, a global minimal solution is an Agmon ground state of the corresponding critical operator $-\Delta_p +\omega$. Moreover, $-\Delta_p +\omega$ is critical in $\R^N$ if and only if it admits a null-sequence in $\R^N$, i.e., a nonnegative sequence $(\phi_n) \in W^{1,p}(\R^N) \cap C_c(\R^N)$   satisfying the following:
		\begin{itemize}
			\item there exists a subdomain $O \Subset \R^N$ such that $\|\phi_n\|_{L^p(O)}  \asymp 1$ for all $n \in \N,$ and
			\item $\displaystyle\lim_{n \rightarrow \infty} \left[ \int_{\R^N} (|\nabla \phi_n|^p + \omega |\phi_n|^p) \dx\right]=0$. 
		\end{itemize} Any null-sequence converges weakly in $L^{p}_\loc(\Gw)$ to the unique (up to a multiplicative constant) positive (super)solution of the equation $[-\Delta_p +\omega](w)=0$ in $\R^N$, hence, it converges to the ground state.  Furthermore,   there exists a null-sequence which converges locally uniformly in $\R^N$ to the ground state \cite{Yehuda_Georgios}.

\medskip

\noindent The following proposition plays a crucial role in the proof of our main results.

\begin{prop} \label{prop-Lp}
Let $p<N$ and let $V\in L^1(\R^N)\cap L^{\infty}(\R^N)$  be a critical potential for the $p$-Laplacian in $\R^N$ satisfying the following Fuchsian type behavior at infinity
\begin{equation}\label{eq_Fuchs}
|V(x)|\leq  C\jap x^{-p} \quad x\in \R^N, \qquad \mbox{where } \jap x:=\sqrt{1+|x|^2}\, . 
\end{equation}
Then the ground state $\phi_0$ of 
the operator $-\Delta_p + V$ satisfies
$\phi_0 \in L^p(\R^N)$. Moreover,  
$$\phi_0(x) \asymp\jap x^{(p-N)/(p-1)} \qquad  x\in \R^N.$$
In particular, $\phi_0 \in L^p(\R^N)$ if and only if $p^2<N$. 
\end{prop}
\begin{proof}
It follows from  \cite[Theorem~1.17]{fp}
 that the Agmon ground state $\phi_0$ satisfies 
\begin{equation}\label{eq_est}
\phi_0(x) \asymp \jap x^{(p-N)/(p-1)} , 
\end{equation}
 see also \cite[Theorem~1.1]{CD} and  \cite{ADS}. 
 
\end{proof}

\section{\bf The case $p^2< N$}
\label{sec_minimizer}
In this section we give the proof of Theorem \ref{thm-linear}. 
The next result ensures the existence of a minimizer for $\ga \geq0$ in case  $p^2<N$. 

\begin{lem} \label{lem_mimimizer}
Let $p^2<N$ and $V\in L^1(\R^N)\cap L^{\infty}(\R^N)$, be a critical potential for the $p$-Laplacian in $\R^N$ satisfying \eqref{eq_Fuchs}. Assume that $W \in C_c(\R^N)$ satisfies  \eqref{condition}. Then for any  $\ga \geq  0$ there is a positive function $\phi_\ga\in W^{1,p}(\R^N)$ such that
\begin{equation} \label{minimum}
\lambda(\alpha)= \frac{Q_{\alpha W}[\phi_\ga]}{\|\phi_\ga\|_p^p}.
\end{equation}
Moreover, $Q_{\ga W-\gl(\ga)}$ is critical and $\phi_\ga$ is an Agmon ground state.
\end{lem}
\begin{proof}
By our assumptions and \cite[Prop.~4.5]{pt} we have $\lambda(\alpha)\leq 0$ for any $\ga\geq 0$ and $\lambda(\alpha) = 0$  if and only if $\ga=0$. Since $W\in C_c(\R^N)$,  in light of Proposition \ref{prop-Lp} we have 
\begin{equation} 
\lambda_\infty(\alpha):= \lim_{R\to\infty}\  \inf_{u\in C_0^\infty( B^c_R(0))} \frac{Q_{\alpha W}[u]}{\|u\|_p^p} =  0
\end{equation}
  for any $\ga\geq 0$. Hence, for $\ga>0$ the functional $Q_{\alpha W}$ has a {\it spectral gap} and therefore the operator 
$$-\Delta_p \varphi + V|\varphi|^{p-2}\varphi -\alpha W |\varphi|^{p-2}\varphi -\lambda(\alpha) |\varphi|^{p-2}\varphi$$
 is critical in $\R^N$ and admits an Agmon ground state $\phi_{\alpha}$. {The proof of this statement is similar to the proof of \cite[Lemma~2.3]{LP}, and therefore it is omitted}. Note that, in order to establish \eqref{minimum} it is enough to prove that $\phi_{\alpha} \in W^{1,p}(\R^N)$ for any $\ga \geq 0$. Recall that $\phi_0$, the Agmon ground state, of $Q_0$  satisfies \eqref{eq_est}, in particular, $\phi_0$ is $L^{p}$-integrable. Moreover, for any $\ga > 0$ the ground state $\phi_0$  is a positive supersolution of the equation 
\begin{equation}\label{eq_PDE_alpha}
-\Delta_p \varphi + V |\varphi|^{p-2}\varphi-\alpha W |\varphi|^{p-2}\varphi = \lambda(\alpha)|\varphi|^{p-2}\varphi
\end{equation}
near infinity, as $W$ is compactly supported and $\lambda(\alpha) <0$. 
Recall that the ground state $\phi_{\alpha}$ is a positive solution of minimal growth at infinity of \eqref{eq_PDE_alpha}. Therefore, there exists $C>0$ and $R$ sufficiently large and independent on $\ga$ such that $\phi_{\alpha} \leq C \phi_0 \asymp\jap x^{(p-N)/(p-1)}$ in $\R^N\setminus B_R(0)$. Hence, $\phi_{\alpha} \in L^p(\R^N)$.

Next we show that in fact  $\phi_{\alpha} \in W^{1,p}(\R^N)$ for $\ga \geq 0$.
Since $\phi_{\alpha}$ is a ground state, there exists a {\it null-sequence} $(\varphi_{\alpha,n})$ in $C_c^{\infty}(\R^N)$ such that $0 \leq \varphi_{\alpha,n} \leq \phi_{\alpha}$ \cite[Remark 5.4 $\&$ Lemma 5.5.]{DP_JAM} and $\varphi_{\alpha,n} \rightarrow \phi_{\alpha}$ in $L^p_{\loc}(\R^N)$.  Consequently, 
\begin{align} \label{Comp:1}
\|\varphi_{\alpha,n}\|_{W^{1,p}(\R^N)}^p &= Q_{\ga W-\lambda(\alpha)}(\varphi_{\alpha,n}) - \int_{\R^N} \big(V- \alpha W - \lambda(\alpha)\big)|\varphi_{\alpha,n}|^p \dx + \int_{\R^N} |\varphi_{{\alpha,n}}|^p \dx \nonumber \\
& \leq Q_{\ga W -\lambda(\alpha)}(\varphi_{{\alpha,n}}) \!+ \!\!\int_{\R^N}\!\! V^-|\phi_{\alpha}|^p \!\dx \!+\! \alpha \!\int_{\R^N} \!\!|W| |\phi_{\alpha}|^p \!\dx \!+\! \int_{\R^N} |\phi_{\alpha}|^p \dx\,. 
\end{align}
Since $(\varphi_{\alpha,n})$ is a null-sequence, it follows that $Q_{{\alpha W} -\lambda(\alpha)}(\varphi_{\alpha,n}) \rightarrow 0$ as $n \rightarrow \infty$. Also, since $V\in L^{N/p}(\R^N)$,  $\phi_{\alpha} \leq C \phi_0 \asymp\jap x^{(p-N)/(p-1)}$,   then by  H\"{o}lder inequality it follows that $\phi_{\alpha} \in L^p(\R^N,V^-)$. Thus, \eqref{Comp:1} implies that $(\varphi_{\alpha,n})$ is bounded in $W^{1,p}(\R^N)$. Due to the reflexivity of $W^{1,p}(\R^N)$, up to a subsequence, there exists $\psi_{\alpha} \in W^{1,p}(\R^N)$ such that $\varphi_{\alpha,n} \rightharpoonup \psi_{\alpha}$ in $W^{1,p}(\R^N)$. Consequently, by Rellich-Kondrachov compactness theorem, up to a subsequence, $\varphi_{\alpha,n} \rightarrow \psi_{\alpha}$ in $L^p_{\loc}(\R^N)$. Hence, $\psi_{\alpha}=c_\ga\phi_{\alpha}$ for some constant $c_\ga>0$. This implies $\phi_{\alpha} \in W^{1,p}(\R^N)$. 
\end{proof}

\begin{proof}[\bf Proof of Theorem \ref{thm-linear}]
By Lemma~\ref{minimum}, we have $\phi_0 \in W^{1,p}(\R^N)$.
Next, using $u=\phi_0$ as a test function in the Rayleigh quotient \eqref{lambda-def} with $\ga>0$, we get
\begin{equation} \label{ineq}
\lambda(\alpha) \leq -\frac{\alpha\int_{\R^N} W (\phi_0)^p\dx}{\|\phi_0\|_p^p}\, < 0.
\end{equation}
Hence, $\lambda(\alpha) <0$ for all $\ga>0$ (this in fact, follows also from \cite[Prop.~4.5]{pt}). Consequently,

\begin{equation} \label{upperb}
\limsup_{\alpha\to 0_+} \alpha^{-1}\lambda(\alpha) \leq -\frac{\int_{\R^N} W (\phi_0)^p\dx}{\|\phi_0\|_p^p}\, < 0.
\end{equation}
To prove the lower bound, recall that by Lemma~\ref{lem_mimimizer} for any $\alpha>0$ there exists  $0< \phi_\alpha\in W^{1,p}(\R^N)$ such that 
$$
\lambda(\alpha) = \frac{Q_{{{\alpha}W}}[\phi_\alpha]}{\|\phi_{\alpha}\|_p^p}\, .
$$ 
This implies
\begin{equation} \label{lowerb-1}
 \lambda(\alpha)  \geq -\alpha\ \frac{ \int_{\R^N} W (\phi_\alpha)^p\dx}{\|\phi_\alpha\|_p^p}\, .
\end{equation} 
We may assume that $\phi_\ga(0)=1$. Let $\ga \searrow 0$, then $\lambda(\alpha)\nearrow 0$. The Harnack convergence principle \cite[Proposition 2.11]{Yehuda_Georgios} and the uniqueness of a positive solution of the critical equation $-\Delta_p \varphi + V|\varphi|^{p-2}\varphi=0$ in $\R^N$  satisfying $\varphi(0)=1$, imply that $\phi_\alpha \to \phi_0$ in $L^\infty_{\loc}(\R^N)$, and therefore, 
$$
\lim_{\alpha\to 0+}  \int_{\R^N} W (\phi_\alpha)^p\dx = \int_{\R^N} W (\phi_0)^p\dx.
$$
Now let $O= \R^N\setminus {\rm supp} \,W$. Then there exists $C>0$ such that for any $0<\alpha\leq 1$   
\begin{equation} \label{u-bc}
C^{-1}\leq \phi_\alpha \big|_{\partial O}  \leq C.
\end{equation}
Moreover, $\phi_\alpha$ is a  positive solution of the equation $-\Delta_p \varphi + V|\varphi|^{p-2}\varphi = \lambda(\alpha) |\varphi|^{p-2}\varphi$ in $O$ of minimal growth at infinity. Since $\phi_0$ is a positive supersolution of the 
same equation, it follows that ${\phi_\alpha}\leq C\phi_0$ in $\R^N$, where $C>0$ is a constant independent of $\ga$, and hence, by the dominated convergence, $\phi_\alpha \to \phi_0$ in $L^p(K)$. This in combination with $\phi_\alpha \to \phi_0$ in $L^\infty_{\loc}(\R^N)$ and 
\eqref{lowerb-1} implies 
$$
\hspace{5cm} \liminf_{\alpha\to 0_+} \alpha^{-1}\lambda(\alpha) \geq -\frac{\int_{\R^N} W (\phi_0)^p\dx}{\|\phi_0\|_p^p}\, .  \hspace{5.5cm}\qedhere
$$
\end{proof}
\section{\bf The case $p<N\leq p^2$}
\label{sec_lsc}
Similarly as in the case $N>p^2$ we start by showing that the variational problem \eqref{lambda-def} admits a minimizer for $\ga>0$. To this end, we need the following lemma.

\begin{lem} \label{lem-semicont} 
Suppose that  $V,W\in L^{N/p}(\R^N)\cap L^s_\loc(\R^N)$ for some $s>N$. Then the functional $Q_{\alpha W}$  is weakly lower
semicontinuous in $W^{1,p}(\R^N)$.
\end{lem}
\begin{proof} It will be convenient to denote 
\begin{equation} \label{V-tilde} 
V_\ga  = V -\alpha W.
\end{equation}
Assume that $(u_j)$ converges weakly in $W^{1,p}(\R^N)$ to some $u$. Since
$\|\nabla u\|_p^p$ is weakly lower semi-continuous, it suffices to show that 
\begin{equation} \label{enough} 
\lim_{j\to \infty} \int_{\R^N} V_\ga  (|u_j|^p-|u|^p)\dx  =0\, .
\end{equation}
Pick $q$ such that $p<q<p^*$ and denote by  $q'$ the H\"older conjugate of $q$. Let
$$
f_j := \frac{|u_j|^p - |u|^p}{|u_j|-|u| }\, . 
$$
The sequence $(u_j)$ is bounded in $W^{1,p}(\R^N)$. Hence from the Sobolev inequality \eqref{sobolev} it follows that 
\begin{equation} \label{uj-Lr}
\sup_j \|u_j\|_r < \infty \qquad \forall\ r\in [\, p, \, p^*]\, .
\end{equation}
Note also that  
\begin{equation} \label{fj-upperb}
|f_j|\leq p\max\{|u_j|^{p-1},|u|^{p-1}\}\, .
\end{equation} 
Let  
$t:= {p(q-1)}/{q(p-1)}$.
Since the mapping $x\mapsto \frac{x}{x-1}$ is strictly decreasing on $(1,\infty)$, it follows that $t>1$, and in view of \eqref{uj-Lr}, \eqref{fj-upperb},
\begin{equation} \label{fj-sup}
\sup_j\, \|f_j\|_ {L^{q't} (\R^N)} \, < \, \infty\, . 
\end{equation}
Now, for any $R>0$ we have, by H\"older {inequality,}
\begin{align} \label{eq:aux1}
\Big | \int_{\R^N} V_\ga  (|u_j|^p-|u|^p)\dx \Big |  & \leq
\|u_j - u\|_{L^q(B_R)} \|V_\ga   f_j\|_{L^{q'}(B_R)}  
+ 2\|V_\ga \|_{L^{N/p}({B_R^c})}\,  \sup_j \|u_j\|_{p^*}\\
&\leq  \|u_j - u\|_{L^q(B_R)} \, \|V_\ga \|_{L^{q' t'} (B_R)}\, \|f_j\|_ {L^{q't} (\R^N)}+ 2\|V_\ga \|_{L^{N/p}({B_R^c})}\,  \sup_j \|u_j\|_{p^*}  \nonumber \, ,
\end{align}
where $B_R$ denotes the ball of radius $R$ centered in $0$, and where $t'$ is the H\"older conjugate of $t$. 
By the Rellich-Kondrashov theorem, see e.g.~\cite[Thm.~8.9]{LL}, up to a subsequence, $u_j$
converges to $u$ in $L^q_\loc$ for any $p\leq q < p*$. Since 
$$
q' t' = \frac{qp}{q-p}\, \qquad \text{and} \qquad \frac{p^* p}{p^*-p} =N\, ,
$$
by taking $q$ sufficiently close to $p^*$ we can make sure that $q' t' = s >N$. 
Then by sending first $j\to \infty$ and then $R\to
\infty$ in \eqref{eq:aux1} we obtain \eqref{enough} and hence the claim.
\end{proof}

\begin{lem} \label{lem-minimizer} 
Let $p<N \leq p^2$. Assume that $V$ and $W$ satisfy {the hypotheses of Lemma \ref{lem-semicont}.} In addition, suppose that $V\in L^q(\R^N)$ for some $N/p <q < p^*$.
Let $\ga>0$ such that  $\lambda(\alpha)<0$. Then there is a positive function $\phi_\ga\in W^{1,p}(\R^N)$ such that
\begin{equation} \label{minimum1}
\lambda(\alpha)= \frac{Q_{\alpha W}[ \phi_\ga]}{\|\phi_\ga\|_p^p}.
\end{equation}
Moreover, $Q_{\ga W-\gl(\ga)}$ is critical and $\phi_\ga$ is an Agmon ground state. 
\end{lem}

\begin{proof}
Let $(u_j)$ be a minimizing sequence for $Q_{\alpha W}$, normalized such that $\|u_j\|_p=1$ for any $j\in\N$. On the other hand,
the Sobolev inequality \eqref{sobolev} implies that $ u_j\in L^{p^*}(\R^N)$ for all  $j\in\N$. Let
\begin{equation} \label{theta}
\theta= \frac{N-p}{p(q-1)}  \, \in (0,1) . 
\end{equation} 
Then
\begin{equation} \label{r-def}
r :=  \frac{pq}{q-1}= \theta p^* +(1-\theta) p\, .
\end{equation}
From the H\"older inequality and from {the Sobolev inequality} \eqref{sobolev}, we thus get
\begin{equation} \label{interpolation}
\|u_j\|_r^r \, \leq \, \|u_j\|_{p^*}^{\theta p*}\, \leq C\, \|\nabla u_j\|_p^{\theta p*}\  \qquad \forall\, j\in\N, 
\end{equation}
with $C$ independent of $j$. The hypothesis 
 $\lambda(\alpha)<0$ allows to assume, without loss of generality, that $Q_{\alpha W}[ u_j]<0$ for any $j\in\N$. H\"older inequality combined with \eqref{interpolation}
 now gives 
 \begin{align*}
 \|\nabla u_j\|_p^p\, < \, \int_{\R^N} |V_\ga | \, |u_j|^p\dx \leq \| V_\ga \|_q\, \|u_j\|_r^{\frac{r(q-1)}{q}}\, \leq\, C \| V_\ga \|_q\, \|\nabla u_j\|_p^{{\tilde\theta} p^*}\, ,
 \end{align*}
 where ${\tilde\theta}= \frac{\theta(q-1)}{q}$. Since
 $$
 {\tilde\theta} p^* =  \frac{\theta(q-1)}{q} \frac{N p}{N-p}= \frac Nq <p,
 $$
 in view of \eqref{theta} and the assumption $q > \frac Np$, it follows that the sequence $(u_j)$ is bounded in $W^{1,p}(\R^N)$. Therefore, there exists a subsequence, which we 
 continue to denote by $(u_j),$ converging weakly in  $W^{1,p}(\R^N)$ to some $u_\ga$. The weak convergence implies
$$
\|u_\ga\|_p\leq \liminf_{j\to\infty} \|u_j\|_p = 1.
$$
Since $Q_{\alpha W}$ is weakly lower semicontinuous by Lemma
\ref{lem-semicont}, we deduce that 
$$
0>\lambda( \alpha) = \lim_{j\to\infty}\, Q_{\alpha W}[ u_j] \geq
Q_{\alpha W}[ u_\ga] \geq \lambda( \alpha)\, \|u_\ga\|_p^p \geq
\lambda( \alpha).
$$
Hence, $Q_{\alpha W}[ u_\ga] = \lambda(\alpha), \ \|u_\ga\|_p=1,$ and \eqref{minimum1} follows. Finally, since $|\,\nabla |u_\ga|\,| = |\nabla u_\ga|$ almost everywhere, we may choose $u_\ga\geq 0$. The Harnack 
inequality then implies $u_\ga>0$. The criticality of $Q_{\ga W-\gl(\ga)}$ follows immediately since  $\{u_\ga\}$ is a null-sequence for $Q_{\ga W-\gl(\ga)}$ and therefore, $u_\ga=\phi_\ga$ is the corresponding Agmon ground state.
\end{proof}

\begin{rem}
{\em Note that the hypothesis $N \leq p^2$ implies  $N/p  < p^*$, which makes the choice of $q$ feasible in the above lemma}. 
\end{rem}
\vspace{.1cm}

 \subsection{\bf Lower bounds}   \label{sec_lb}
 It remains to study the asymptotic of $\gl(\ga)$ as $\alpha \rightarrow 0$ when $p<N \leq p^2$. 
The following proposition shows that if  $V,W$ have compact supports, then the speed at which $\gl(\ga)$ tends to $0$ is faster than linear.

\begin{prop} \label{Prop:conv_la_al}
	Let $1<p<N \leq p^2$ and let $V,W \in C_c(\R^N)$ such that $V$ is critical in $\R^N$. Then $\frac{\lambda(\alpha)}{\alpha} \rightarrow 0$ as $\alpha \rightarrow 0_+$. 
\end{prop}

\begin{proof}
	From Lemma \ref{lem-minimizer}, we know that $\lambda(\alpha)$ is achieved at $\phi_{\alpha} \in W^{1,p}(\R^N)$, i.e., 
	\begin{align} \label{Eq:lam_al}
	\lambda(\alpha)=\frac{Q_{\alpha W}(\phi_{\alpha})}{\|\phi_{\alpha}\|_p^p} \geq -\alpha \frac{\int_{\R^N} W|\phi_{\alpha}|^p \dx}{\|\phi_{\alpha}\|_{p}^p}\,,
	\end{align}
	where $\phi_{\alpha}>0$ and can be chosen satisfying $\phi_{\alpha}(0)=1$.
	Using the arguments as in the proof of Theorem \ref{thm-linear}, it follows that $\phi_{\alpha} \rightarrow \phi_0$ in $L^{\infty}_{\loc}(\R^N)$, where $\phi_0$ is an Agmon ground state of the critical operator $-\Delta_p+V$. Since $W$ has compact support, it follows that $\lim_{\alpha \rightarrow 0} \int_{\R^N} W|\phi_{\alpha}|^p \dx = \int_{\R^N} W|\phi_0|^p \dx$. Also, \eqref{eq_est} implies that $\phi_0\not \in L^p(\R^N)$ as $p<N \leq p^2$. Thus, it follows from Fatou's lemma that $\liminf_{n \rightarrow \infty} \|\phi_{\alpha}\|_{L^p(\R^N)}^p =\infty$. From \eqref{Eq:lam_al}, we get
	$$ 0 > \frac{\lambda(\alpha)}{\alpha} \geq - \frac{\int_{\R^N} W|\phi_{\alpha}|^p \dx}{\|\phi_{\alpha}\|_{p}^p} \,.$$
	Consequently, the proposition follows.
\end{proof}

\noindent In Proposition \ref{prop:LB1} and \ref{prop:LB2} below we will prove the 
 necessary lower bounds for two different ranges of $p$ which together cover the whole interval $(1, N)$.  The main ingredient of the proof is a pointwise lower bound 
on $\phi_\alpha$ established in Lemmas \ref{Lem:LEst1} and \ref{Lem:LEst2}. The case $p^2=N$ is treated separately in Proposition \ref{prop:LB3}. 

\begin{prop} \label{prop:LB1}
Let $2-\frac{1}{N}\leq p<N \leq p^2$ and $V,W \in C_c(\R^N)$ such that $V$ is critical in $\R^N$. Then there exists {a $W$-independent} constant $C=C(p,N,V)>0$  such that 
$$ \liminf_{\alpha \rightarrow 0_+} \alpha^{-\frac{p(p-1)}{N-p}}\lambda(\alpha)  \geq -C \left(\int_{\R^N} W|\phi_0|^p \dx\right)^{{\frac{p(p-1)}{N-p}}}  \,.$$
 \end{prop}

 \noindent To prove this proposition we need the following lemma, which provides a pointwise lower bound of the minimizer $\phi_{\alpha}$ near infinity. This will enable us to use comparison techniques.

\begin{lem} \label{Lem:LEst1}
Let $1< p<N \leq p^2$ and $V,W \in C_c(\R^N)$ such that $V$ is critical in $\R^N$ and the support of $V,W$ are contained inside $B_R$ for some $R>>1$.  For $\ga>0$ let $\phi_{\alpha} \in W^{1,p}(\R^N)$ be a minimizer of $\lambda(\alpha)$ with $\phi_{\alpha}>0$ and $\phi_{\alpha}(0)=1$.  Then there exists an $\alpha,W$-independent constant $C(V,N,p)>0$ such that  
\begin{align}
  \phi_{\alpha} \geq  C   v_{\alpha} \qquad \mbox{on} \ \ B_{R}^c \,,
\end{align}
for all $\alpha>0$ sufficiently small, where $v_{\alpha} \in W^{1,p}(\R^N)$ is a radial, and radially decreasing function such that $$v_{\alpha} = |x|^{-\nu_1}\exp\left(-\left(\frac{\lambda(\alpha)}{1-p}\right)^{1/p}|x|\right) \ \ \mbox{in} \ \ B_R^c \,, \ \  \ \ \nu_1 =\frac{N-1}{p-1} .$$
\end{lem}

\begin{proof}
{Recall that the ground state} $\phi_{\alpha}$ satisfies the equation
\begin{align*}
-\Delta_p  \phi_{\alpha} + V \phi_{\alpha}^{p-1} - \alpha W \phi_{\alpha}^{p-1} -\lambda(\alpha) \phi_{\alpha}^{p-1} =0 \ \ \mbox{in} \ \R^N , 
\end{align*}
{and since} $V,W$ have compact supports inside $B_R$, we have 
\begin{align*}
-\Delta_p  \phi_{\alpha} -\lambda(\alpha) \phi_{\alpha}^{p-1} =0  \qquad \mbox{in } B_R^c
\end{align*}
for every $\ga >0 $.
Consider the given radial, and radially decreasing function $v_{\alpha} \in W^{1,p}(\R^N)$ (cf. \cite[Theorem~{1.1}]{LZ}).
Recall that the formal radial $p$-Laplacian is given by
\begin{equation}\label{eq:green}
-\Delta_p (v) =-\frac{1}{r^{N-1}}\left( r^{N-1}|v'|^{p-2} v'  \right)' =
-|v'|^{p-2}\left[(p-1)v''+\frac{N-1}{r}v'\right].
\end{equation}
Denoting $\mu_{\alpha}=\big(\frac{\lambda(\alpha)}{1-p}\big)^{1/p}$, a direct computation (cf. \cite[Lemma 5.8]{AMM}) shows that    
\begin{align}
-\Delta_p v_{\alpha} & = (1-p)\mu_{\alpha}^{p} \left(1+\frac{\nu_1}{\mu_{\alpha} |x|} \right)^{p-2} v_{\alpha}^{p-1} + \mu_{\alpha}^{p} \left(1+\frac{\nu_1}{\mu_{\alpha} |x|} \right)^{p-2} \left( \frac{A_{\nu_1}}{\mu_{\alpha}|x|}  +  \frac{B_{\nu_1}}{\mu_{\alpha}^2|x|^2} \right) v_{\alpha}^{p-1}  \nonumber  \\
& = \mu_{\alpha}^{p} \left(1+\frac{\nu_1}{\mu_{\alpha} |x|} \right)^{p-2} \left[\frac{A_{\nu_1}}{\mu_{\alpha}|x|} + \frac{B_{\nu_1}}{\mu_{\alpha}^2|x|^2} -(p-1)  \right] v_{\alpha}^{p-1} \nonumber \\
& = \lambda(\alpha) \left(1+\frac{\nu_1}{\mu_{\alpha} |x|} \right)^{p-2} \left[1-\frac{A_{\nu_1}}{\mu_{\alpha}(p-1)|x|} -\frac{B_{\nu_1}}{\mu_{\alpha}^2(p-1)|x|^2}   \right] v_{\alpha}^{p-1}   \nonumber \\
& = \lambda(\alpha) \left(1+\frac{\nu_1}{\mu_{\alpha} |x|} \right)^{p-2} \left[1+\frac{\nu_1}{\mu_{\alpha}|x|} -\frac{B_{\nu_1}}{\mu_{\alpha}^2(p-1)|x|^2}   \right] v_{\alpha}^{p-1}   \nonumber \\
& \leq  \lambda(\alpha) \left(1+\frac{\nu_1}{\mu_{\alpha} |x|} \right)^{p-1}  v_{\alpha}^{p-1}  \ \ \mbox{in} \ B_R^c \,, \label{Eq:1'}
\end{align}
where $A_{\nu_1}=(N-1)-2\nu_1 (p-1)=1-N<0$ and $B_{\nu_1}=\nu_1(N-p-\nu_1(p-1))=1-N <0 $. 
Subsequently, from \eqref{Eq:1'}, we infer that
$$ -\Delta_p v_{\alpha} - \lambda(\alpha) v_{\alpha}^{p-1} \leq 0  \ \  \mbox{in} \ \ B_{R}^c $$
for $\alpha>0$.
From the above discussion, we conclude that  
\begin{align*}
-\Delta_p  v_{\alpha}  - \lambda(\alpha) v_{\alpha}^{p-1} \leq 0 \leq -\Delta_p  \phi_{\alpha}  - \lambda(\alpha) \phi_{\alpha}^{p-1} \qquad \mbox{in} \ B_{R}^c  
\end{align*}
for $\alpha>0$.  Our aim is now to apply the comparison principle \cite[Theorem B.1 \& Lemma B.2]{AMM} to obtain a lower bound on $\phi_\alpha$ near infinity. To do so we need to compare the  
functions $\phi_\alpha$ and $v_\alpha$ on $\partial B_{R}$.
Clearly, 
$$ 
v_{\alpha} = R^{-\nu_1}\exp\Big(-\left(\frac{ \lambda(\alpha)}{1-p}\right)^{1/p}R\Big)  \leq R^{-\nu_1} \qquad \mbox{on}   \ \partial B_{R} 
$$
for all $\alpha >0$. Next we  find a constant $C(p,N,V)>0$ (independent of $\alpha$ and $W$) such that $ v_{\alpha} \leq C \phi_{\alpha}$ in $\partial B_{R}$ for all $\
\alpha>0$ sufficiently small. Recall that $\phi_0$ satisfies \eqref{eq_est}. Hence there exists $M_1,M_2>0$, independent of $\alpha$ and $W$,  such that $$\frac{M_1}{(1+|x|)^{\frac{N-p}{p-1}}} \leq \phi_0(x) \leq \frac{M_2}{(1+|x|)^{\frac{N-p}{p-1}}} \ \ \mbox{in} \ \ \R^N \,.$$ Since $\phi_{\alpha} \rightarrow \phi_0$ in $L^{\infty}_{\loc}(\R^N)$, it follows that $$\phi_{\alpha}(x)  \geq \frac{M_1}{2R^{\frac{N-p}{p-1}}} \qquad \mbox{on}   \ \partial B_{R} $$
uniformly for sufficiently small $\alpha$. Thus, by taking $\alpha >0$ sufficiently small, we have $$v_{\alpha} \leq \frac{2}{M_1} \phi_{\alpha}  \qquad \mbox{on}   \ \partial B_{R} \, .
$$
The comparison principle \cite[Theorem B.1 \& Lemma B.2]{AMM} now ensures that $v_{\alpha} \leq C \phi_{\alpha}$ in $ B_{R}^c$ for all $\alpha >0$ small enough, with a constant  $C(V,N,p)>0$ independent of $\alpha,W$.
\end{proof}

\begin{proof}[\bf Proof of Proposition \ref{prop:LB1}]
By Lemma \ref{lem-minimizer}, for $\ga>0$ there exists $\phi_{\alpha} \in W^{1,p}(\R^N)$ with $\phi_{\alpha}>0$ and $\phi_{\alpha}(0)=1$ such that 
\begin{align} \label{Eq:lam_al_2}
\lambda(\alpha)=\frac{Q_{\alpha W}(\phi_{\alpha})}{\|\phi_{\alpha}\|_p^p} \,.
\end{align}
Moreover, as an Agmon ground state  $\phi_{\alpha}$ satisfies the equation:
\begin{align*}
-\Delta_p  \phi_{\alpha} + V \phi_{\alpha}^{p-1} - \alpha W \phi_{\alpha}^{p-1} -\lambda(\alpha) \phi_{\alpha}^{p-1} =0 \ \ \mbox{in} \ \R^N \,. 
\end{align*}
As $V,W$ have compact supports inside $B_R$ for some $R>>1$, $\phi_{\alpha}$ is a positive (super)solution of $-\Delta_p \varphi - \lambda(\alpha) |\varphi|^{p-2}\varphi =0$ in $B_R^c$.
 For 
 $$
 \nu_0 = \frac{N-p}{p-1}, 
 $$
 consider the function
\begin{align} \label{Eq:valpha}
{w}_{\alpha}:= |x|^{-\nu_0}\exp\left(-\mu_{\alpha}|x|\right) \ \ \mbox{in} \ \ \R^N{\setminus \{0\}} \,,
\end{align}
where $\mu_{\alpha}=\left(\frac{2\lambda(\alpha)}{1-p}\right)^{1/p}$. Observe that ${w}_{\alpha} \in W^{1,p}(\R^N)$ and that it is radially decreasing. A direct computation (cf. \cite[Lemma 5.8]{AMM} or use \eqref{eq:green}) shows that    
\begin{align}
-\Delta_p {w}_{\alpha} & = (1-p)\mu_{\alpha}^{p} \left(1+\frac{\nu_0}{\mu_{\alpha} |x|} \right)^{p-2} {w}_{\alpha}^{p-1} + \mu_{\alpha}^{p} \left(1+\frac{\nu_0}{\mu_{\alpha} |x|} \right)^{p-2} \frac{A_{\nu_0}}{\mu_{\alpha}|x|} {w}_{\alpha}^{p-1} \nonumber  \\
& = \mu_{\alpha}^{p} \left(1+\frac{\nu_0}{\mu_{\alpha} |x|} \right)^{p-2} \left[\frac{A_{\nu_0}}{\mu_{\alpha}|x|} -(p-1)  \right] {w}_{\alpha}^{p-1} \nonumber \\
& = \lambda(\alpha) \left(1+\frac{\nu_0}{\mu_{\alpha} |x|} \right)^{p-2} 2 \left[1-\frac{A_{\nu_0}}{\mu_{\alpha}(p-1)|x|}   \right] {w}_{\alpha}^{p-1}  \ \ \ \ \mbox{in} \ \ \R^N{\setminus \{0\}} \,, \label{Eq:1}
\end{align}
where $A_{\nu_0}=(N-1)-2\nu_0 (p-1)$. Note that we can find $L>>1$ independent of $\alpha$ such that  
$$ \left(1+\frac{\nu_0}{\mu_{\alpha} |x|} \right)^{p-2} 2\left[1-\frac{A_{\nu_0}}{\mu_{\alpha}(p-1)|x|}   \right] \geq 1 \ \ \ \  \  \ \ \mbox{if } \ \ {|x| > \frac{L}{\mu_{\alpha}}}:=R_{\alpha} \,.$$
Clearly,  $R_{\alpha} \rightarrow \infty$ as $\alpha \rightarrow 0$ (since $\mu_{\alpha} \rightarrow 0$ as $\alpha \rightarrow 0$). Recall that $\lambda(\alpha)<0$, therefore, \eqref{Eq:1} implies that
$$ -\Delta_p w_{\alpha} - \lambda(\alpha) w_{\alpha}^{p-1} \leq 0  \ \ \ \  \  \mbox{if} \ \ B_{R_{\alpha}}^c $$
for $\alpha>0$.
From the above discussion, we conclude that  
\begin{align*}
-\Delta_p w_{\alpha}  -\lambda(\alpha) w_{\alpha}^{p-1} \leq 0 \leq -\Delta_p \phi_{\alpha}    -\lambda(\alpha) \phi_{\alpha}^{p-1}  \qquad \mbox{in} \ \ B_{R_{\alpha}}^c   
\end{align*}
for $\alpha>0$ sufficiently small. Using \cite[Lemma A.IV]{BL}, we get
\begin{align*}
w_{\alpha} \leq C(N,p) \frac{\|w_{\alpha}\|_{L^s(\R^N)}}{|x|^{\frac{N-1}{p-1}}} \ \ \ \mbox{in} \ \ \R^N\setminus \{0\} \,,
\end{align*}
where $s=\frac{N(p-1)}{N-1} \geq 1$ as $p\geq 2-\frac{1}{N}$. Since $\|w_{\alpha}\|_{L^s(\R^N)} \leq C(N,p) \mu_{\alpha}^{N(N-p)/(N-1)p}$ for all $\alpha$, it implies that $\|w_{\alpha}\|_{L^s(\R^N)} \leq C(N,p)$ for all $\alpha >0$ sufficiently small (as $\mu_{\alpha} \rightarrow 0$ when $\alpha \rightarrow 0$ and $p<N$). Thus, it  follows that there exists $C(N,p)>0$ such that
$$ C(N,p) \exp(-L) w_{\alpha} \leq R_{\alpha}^{-\nu_1} \exp(-L) = R_{\alpha}^{-{\nu_{1}}}\exp\Big(-\left(\frac{\lambda(\alpha)}{1-p}\right)^{1/p} R_{\alpha}\Big)   \qquad \mbox{on}   \ \partial B_{R_{\alpha}} $$
for small $\alpha$, where ${\nu_{1}}=  \frac{N-1}{p-1}$.  By Lemma \ref{Lem:LEst1} 
there exists $C >0$ such that
$$ \phi_{\alpha} \geq C |x|^{-{\nu_{1}}}\exp\Big(-\left(\frac{\lambda(\alpha)}{1-p}\right)^{1/p}|x|\Big) \qquad \mbox{on}   \ B_{R}^c  $$
for all $\alpha >0$ sufficiently small.
Hence, $\phi_{\alpha} \geq  C_1 w_{\alpha}$ on $\partial B_{R_{\alpha}}$, where $C_1>0$ is independent of $\alpha$. Now we apply the comparison principle \cite[Theorem B.1 \& Lemma B.2]{AMM} to ensure that $ \phi_{\alpha} \geq C_1 w_{\alpha}$ in $ B_{R_{\alpha}}^c$ for sufficiently small $\alpha$. Using this, one can estimate
\begin{align*} 
\|\phi_{\alpha}\|_p^p= \int_{\R^N} |\phi_{\alpha}|^p \dx \geq  \int_{B_{R_{\alpha}}^c} |\phi_{\alpha}|^p \dx & \geq C_1^p \int_{B_{R_{\alpha}}^c} |w_{\alpha}|^p \dx \geq C_2^p (-\lambda(\alpha))^{\nu_0-\frac{N}{p}}\int_{B_{L}^c} |\hat{w}|^p \dx  \,,
\end{align*}
for sufficiently small $\alpha$, where \begin{align*} 
\hat{w}:= |x|^{-\nu_0}\exp\left(-2^{1/p}p|x|\right) {\in L^p(\R^N)} \,. 
\end{align*}
Hence, we obtain a positive $\alpha,W$-independent constant $C_2>0$ such that \begin{align} \label{Eq:Lpest}
\|\phi_{\alpha}\|_p^p \geq C_2 (-\lambda(\alpha))^{\nu_0-\frac{N}{p}}
\end{align} 
for all $\alpha>0$ sufficiently small. Using this estimate in \eqref{Eq:lam_al_2}, we get
\begin{align*}
\lambda(\alpha) \geq - \alpha \frac{\int_{\R^N}W |\phi_{\alpha}^p| \dx}{\|\phi_{\alpha}\|_p^p} \geq - \frac{C\alpha}{(-\lambda(\alpha))^{\nu_0-\frac{N}{p}}} \int_{\R^N}W |\phi_{\alpha}^p| \dx \,,
\end{align*}
where $C=C_2^{-1}$.
Now the claim follows because $\int_{\R^N}W |\phi_{\alpha}^p| \dx \rightarrow \int_{\R^N}W |\phi_0^p| \dx$ as $\alpha \rightarrow 0_+$.
\end{proof}

\begin{prop} \label{prop:LB2}
Let $1< p<N \leq p^2$ be such that $p<(N+1)/2$  and $V,W \in C_c(\R^N)$ such that $V$ is critical in $\R^N$. Then there exists {a $W$-independent} $C(p,N,V)>0$  such that 
$$ \liminf_{\alpha \rightarrow 0_+} \alpha^{-\frac{p(p-1)}{N-p}}\lambda(\alpha)  \geq -C \left(\int_{\R^N} W|\phi_0|^p \dx\right)^{{\frac{p(p-1)}{N-p}}}  \,.$$
 \end{prop}
As in the proof of Proposition \ref{prop:LB1}, we first prove a pointwise lower bound on the minimizer $\phi_{\alpha}$ near infinity.

\begin{lem} \label{Lem:LEst2}
Let $1< p<N \leq p^2$ be such that $p<(N+1)/2$  and $V,W \in C_c(\R^N)$ such that $V$ is critical in $\R^N$ and support of $V,W$ are contained inside $B_R$ for some $R>>1$. Assume that $\phi_{\alpha} \in W^{1,p}(\R^N)$ is a minimizer of $\lambda(\alpha)$ with $\phi_{\alpha}>0$ and $\phi_{\alpha}(0)=1$.  Then there exist $\alpha, W$-independent positive constants $C(V,N,p)$ and $\beta$   such that  
\begin{align*}
  \phi_{\alpha} \geq  C   v_{\alpha,\beta} \qquad \mbox{on} \ B_{R}^c \,, 
\end{align*}
 where $v_{\alpha,\beta} \in W^{1,p}(\R^N)$ is a radial function such that $$v_{\alpha,\beta}= |x|^{-\nu_0}\exp\left(-\left(\frac{\lambda(\alpha)}{1-p}\right)^{1/p}\beta |x|\right) \ \ \mbox{in} \ \ B_R^c \,, \ \  \ \ \nu_0=\frac{N-p}{p-1} .$$
\end{lem}

\begin{proof}
{Recall that the ground state} $\phi_{\alpha}$ satisfies the equation
\begin{align*}
-\Delta_p  \phi_{\alpha} + V \phi_{\alpha}^{p-1} - \alpha W \phi_{\alpha}^{p-1} -\lambda(\alpha) \phi_{\alpha}^{p-1} =0 \ \ \mbox{in} \ \R^N , 
\end{align*}
{and since} $V,W$ have compact supports, there exists  $R>1$ such that 
\begin{align*}
-\Delta_p  \phi_{\alpha} -\lambda(\alpha) \phi_{\alpha}^{p-1} =0  \qquad \mbox{in } B_R^c
\end{align*}
for every $\ga >0 $. Now consider the function
\begin{align*} 
v_{\alpha,\beta}(x):= |x|^{-\nu_{0}}\exp\left(-\left(\frac{\lambda(\alpha)}{1-p}\right)^{1/p} \beta |x|\right) \ \ \mbox{in} \ \R^N\setminus \{0\} \,,
\end{align*}
for some $\beta >0$ that will be chosen later. 

Denoting $\mu_{\alpha}=(\frac{\lambda(\alpha)}{1-p})^{1/p}$, a direct computation (cf. \cite[Lemma 5.8]{AMM} or use \eqref{eq:green}) shows that    
\begin{align}
-\Delta_p v_{\alpha,\beta} & = (1-p)\mu_{\alpha}^{p} \beta^p \left(1+\frac{\nu_{0}}{\mu_{\alpha}\beta |x|} \right)^{p-2} v_{\alpha,\beta}^{p-1} + \mu_{\alpha}^{p} \beta^p \left(1+\frac{\nu_{0}}{\mu_{\alpha} \beta |x|} \right)^{p-2}  \frac{A_{\nu_{0}}}{\mu_{\alpha}\beta|x|}   v_{\alpha,\beta}^{p-1}  \nonumber  \\
& = \mu_{\alpha}^{p}\beta^p \left(1+\frac{\nu_{0}}{\mu_{\alpha}\beta |x|} \right)^{p-2} \left[\frac{A_{\nu_{0}}}{\mu_{\alpha}\beta|x|}  -(p-1)  \right] v_{\alpha,\beta}^{p-1} \nonumber \\
& = \lambda(\alpha)\beta^p \left(1+\frac{\nu_{0}}{\mu_{\alpha}\beta |x|} \right)^{p-2} \left[1+\frac{|A_{\nu_{0}}|}{\mu_{\alpha}\beta(p-1)|x|}  \right] v_{\alpha,\beta}^{p-1} \nonumber  \\
& = \lambda(\alpha)\beta^p \frac{|A_{\nu_{0}}|}{\nu_{0}(p-1)} \left(1+\frac{\nu_{0}}{\mu_{\alpha}\beta |x|} \right)^{p-2} \left[\frac{\nu_{0}(p-1)}{|A_{\nu_{0}}|}+\frac{\nu_{0}}{\mu_{\alpha}\beta|x|}  \right] v_{\alpha,\beta}^{p-1} \nonumber   \\
& \leq  \lambda(\alpha)\beta^p \frac{|A_{\nu_{0}}|}{\nu_{0}(p-1)} \left(1+\frac{\nu_{0}}{\mu_{\alpha}\beta |x|} \right)^{p-1}  v_{\alpha,\beta}^{p-1}  \ \ \ \  \ \mbox{in} \ \R^N{\setminus \{0\}} \,, \label{Eq:1'2}
\end{align}
where $A_{\nu_{0}}=(N-1)-2\nu_{0} (p-1)<0$ as $p<\frac{N+1}{2}$. The last inequality uses the fact that $\frac{\nu(p-1)}{|A_{\nu_{0}}|} \geq 1$. Subsequently, by taking $\beta$ large enough in \eqref{Eq:1'2}, we infer that
$$ -\Delta_p v_{\alpha,\beta} - \lambda(\alpha) v_{\alpha,\beta}^{p-1} \leq 0  \ \  \ \ \mbox{in} \ \ B_{R}^c $$
for all $\alpha>0$.
From the above discussion, we conclude that  
\begin{align*}
-\Delta_p  v_{\alpha,\beta}  - \lambda(\alpha) v_{\alpha,\beta}^{p-1} \leq 0 \leq -\Delta_p  \phi_{\alpha}  - \lambda(\alpha) \phi_{\alpha}^{p-1} \qquad \mbox{in} \ B_{R}^c  
\end{align*}
for  $\alpha>0$.  As in the proof of Lemma \ref{Lem:LEst1} we now apply the comparison principle \cite[Theorem B.1 \& Lemma B.2]{AMM} to ensure that there exists an $\alpha,W$-independent $C(p,N,V)>0$ such that $ v_{\alpha,\beta} \leq C \phi_{\alpha}$ in $ B_{R}^c$ for all $\
\alpha>0$ sufficiently small.
Clearly, 
$$ v_{\alpha,\beta} = R^{-\nu_{0}}\exp\Big(-\left(\frac{ \lambda(\alpha)}{1-p}\right)^{1/p}\beta R\Big)  \leq R^{-\nu_{0}} \qquad \mbox{on}   \ \partial B_{R} $$
for all $\alpha >0$. Since $\phi_0$ satisfies \eqref{eq_est}, there exist constants $M_1,M_2>0$, independent of $\alpha$ and $W$,  such that $$\frac{M_1}{(1+|x|)^{\frac{N-p}{p-1}}} \leq \phi_0(x) \leq \frac{M_2}{(1+|x|)^{\frac{N-p}{p-1}}} \ \ \mbox{in} \ \ \R^N \,.$$ Since $\phi_{\alpha} \rightarrow \phi_0$ in $L^{\infty}_{\loc}(\R^N)$, it follows that $$\phi_{\alpha}(x)  \geq \frac{M_1}{2R^{\frac{N-p}{p-1}}} \qquad \mbox{on}   \ \partial B_{R} $$
uniformly for sufficiently small $\alpha$. Thus, by taking $\alpha >0$ sufficiently small, we have $$v_{\alpha,\beta} \leq \frac{2}{M_1} \phi_{\alpha}  \qquad \mbox{on}   \ \partial B_{R} \,.$$ An applicatin of the comparison principle \cite[Theorem B.1 \& Lemma B.2]{AMM} thus ensures that $v_{\alpha,\beta} \leq C \phi_{\alpha}$ in $ B_{R}^c$ for all $\alpha >0$ sufficiently small, where $C>0$ is independent of $\alpha$ and $W$.
\end{proof} 
\begin{proof}[\bf Proof of Proposition \ref{prop:LB2}]
By Lemma \ref{lem-minimizer}, for $\ga>0$ there exists $\phi_{\alpha} \in W^{1,p}(\R^N)$ with $\phi_{\alpha}>0$ and $\phi_{\alpha}(0)=1$ such that 
\begin{align} \label{Eq:lam_al_22}
\lambda(\alpha)=\frac{Q_{\alpha W}(\phi_{\alpha})}{\|\phi_{\alpha}\|_p^p} \,.
\end{align}
Moreover, as an Agmon ground state  $\phi_{\alpha}$ satisfies 
\begin{align*}
-\Delta_p  \phi_{\alpha} + V \phi_{\alpha}^{p-1} - \alpha W \phi_{\alpha}^{p-1} -\lambda(\alpha) \phi_{\alpha}^{p-1} =0 \ \ \mbox{in} \ \R^N \,. 
\end{align*}
As $V,W$ have compact supports inside $B_R$ for some $R>>1$, $\phi_{\alpha}$ is a positive (super)solution of $-\Delta_p \varphi - \lambda(\alpha) |\varphi|^{p-2}\varphi =0$ in $B_R^c$. We have seen in Lemma \ref{Lem:LEst2} that 
there exist $\beta,C(V,N,p) >0$ such that
$$ \phi_{\alpha} \geq C\, v_{\alpha,\beta} \qquad \mbox{on}   \ B_{R}^c  $$
for all $\alpha >0,$ where
$$ v_{\alpha,\beta}= |x|^{-{\nu_{0}}}\exp\Big(-\left(\frac{\lambda(\alpha)}{1-p}\right)^{1/p} \beta |x|\Big) \qquad \mbox{on}   \ B_{R}^c  \,.$$
 Take $R_{\alpha}=R/\mu_{\alpha}$. Then $R_{\alpha} \geq R$ for all $\alpha$ sufficiently small (as $\mu_{\alpha} \rightarrow 0$ when $\alpha \rightarrow 0$). Thus, $\phi_{\alpha} \geq  C v_{\alpha,\beta}$ on $ B_{R_{\alpha}}^c$, where $C>0$ is independent of $\alpha,W$. Using this, one can estimate
\begin{align} 
\|\phi_{\alpha}\|_p^p= \int_{\R^N} |\phi_{\alpha}|^p \dx \geq  \int_{B_{R_{\alpha}}^c} |\phi_{\alpha}|^p \dx & \geq C^p \int_{B_{R_{\alpha}}^c} |w_{\alpha}|^p \dx \geq C_1^p (-\lambda(\alpha))^{\nu_{0}-\frac{N}{p}}\int_{B_{R}^c} |\hat{w}|^p \dx  , \label{Eq1}
\end{align}
for sufficiently small $\alpha$, where \begin{align*} 
\hat{w}:= |x|^{-\nu_{0}}\exp\left(-p|x|\right) {\in L^p(\R^N)} \,.
\end{align*}
Hence, we obtain a positive $\alpha,W$-independent constant $C_2>0$ such that 
\begin{align} \label{Eq:Lpest2}
\|\phi_{\alpha}\|_p^p \geq C_2 (-\lambda(\alpha))^{\nu_0-\frac{N}{p}}
\end{align} 
for all $\alpha>0$ sufficiently small. Using this estimate in \eqref{Eq:lam_al_22}, we get
\begin{align*}
\lambda(\alpha) \geq - \alpha\ \frac{\int_{\R^N}W |\phi_{\alpha}^p| \dx}{\|\phi_{\alpha}\|_p^p} \geq - \frac{C\alpha}{(-\lambda(\alpha))^{\nu_0-\frac{N}{p}}} \int_{\R^N}W |\phi_{\alpha}^p| \dx \,,
\end{align*}
where $C=C_2^{-1}$.
The proposition follows because $\int_{\R^N}W |\phi_{\alpha}^p| \dx \rightarrow \int_{\R^N}W |\phi_0^p| \dx$ as $\alpha \rightarrow 0_+$.
\end{proof}

Note that if $N=p^2$, then $\frac{p(p-1)}{N-p}= {1}$. Thus, the lower bound of Proposition \ref{prop:LB2} {is actually weaker than the estimate given in} Proposition \ref{Prop:conv_la_al} when $N=p^2$. Nevertheless, replacing the crude estimate in \eqref{Eq1} with an improved one, we obtain a better lower bound of $\lambda(\alpha)$ when $N=p^2$.

 \begin{prop} \label{prop:LB3}
 Let $1< p<N = p^2$ and let $V,W \in C_c(\R^N)$ such that $V$ is critical in $\R^N$. Then there exists $C>0$  such that 
$$ \liminf_{\alpha \rightarrow 0_+}\frac{\lambda(\alpha) |\log (\alpha)|}{\alpha} \geq -C \int_{\R^N} W|\phi_0|^p \dx  \,.$$
\end{prop}
\begin{proof} Note that in this case we always have $p<\frac{N+1}{2}$. As we see in the proof of Proposition \ref{prop:LB2}, there exist positive constants $\beta$ and $C(V,N,p)$,
 independent  of  $\alpha,W$, such that
 \begin{align*} 
  v_{\alpha,\beta}(x):= |x|^{-\frac{N-p}{p-1}} \mathrm{exp}\left(-\left(\frac{{\lambda(\alpha)}}{1-p}\right)^{1/p}\beta |x|\right)  \leq C \phi_{\alpha} \ \ \mbox{on} \ B_{R}^c
\end{align*} 
for sufficiently small $\alpha >0$. Using Proposition \ref{Prop:conv_la_al}, we infer that $\lambda(\alpha)\beta^p/(1-p) \leq {\alpha}$. Hence, we have
\begin{align*} 
   |x|^{-\frac{N-p}{p-1}} \mathrm{exp}\left(-{\alpha^{1/p}}|x| \right)  \leq C \phi_{\alpha} \ \ \mbox{on} \ B_{R}^c \,.
\end{align*}
Now we replace the estimate in \eqref{Eq1} by the following one
 \begin{align*}
\|\phi_{\alpha}\|_p^p \geq  \int_{ B_R^c}\phi_{\alpha}^p \dx   \geq        
 \int_R^\infty \exp(-\ga^{1/p} r)r^{-1}\dr  = \Gg(0,\ga^{1/p})  \sim (-\log(\ga^{1/p}) - \gamma ) \,,
 \end{align*}
 as $\ga \to 0$, where the above well known asymptotic formula for the incomplete gamma function can be found in \cite[Equation (6.5.15) and (5.1.11)]{as}, and $\gg$ is the Euler constant which is positive. So, the right hand side of the above estimate is bigger than a positive constant multiple of $|\log \alpha|$.   Using this estimate in \eqref{Eq:lam_al_2}, we get
\begin{align*}
\lambda(\alpha) \geq - \alpha \frac{\int_{\R^N}W |\phi_{\alpha}^p| \dx}{\|\phi_{\alpha}\|_p^p} \geq -C \frac{\alpha}{|\log \alpha|} \int_{\R^N}W |\phi_{\alpha}^p| \dx \,,
\end{align*}
when $\alpha >0$ is small enough.
Hence, the proposition follows by taking $\alpha \rightarrow 0_+$. 
\end{proof}

\begin{rem} 
{\em $(i)$ Observe that for $p=2$ and $N=3$, we have $\frac{p(p-1)}{N-p}=2$. Thus, the lower estimate in Proposition \ref{prop:LB1} corresponds to
$$ \liminf_{\alpha \rightarrow 0_+}\frac{\lambda(\alpha)}{\alpha^2} \geq -C \left(\int_{\R^N} W|\phi_0|^2 \dx \right)^2  \,.$$
Therefore, in the view of \cite{KS}, the lower bound in Proposition \ref{prop:LB2} is sharp. Also, when $p=2$ and $N=4$, it can be verified that the lower bound in Proposition \ref{prop:LB2} is sharp by comparing it with the corresponding result in \cite{KS}.

$(ii)$ Indeed, the lower bounds in Propositions \ref{prop:LB1}, \ref{prop:LB2}, and \ref{prop:LB3} are sharp. This can be seen from the upper bounds that we obtain in the next section; see Propositions \ref{prop: upperb-1} and \ref{prop: upperb-2}.

$(iii)$ Although Propositions \ref{prop:LB1}, \ref{prop:LB2}, and \ref{prop:LB3} have additional restrictions on the values of $p\in [\sqrt{N},N)$, the three theorems together provide a complete picture of the lower bound for $\lambda(\alpha)$ as $\alpha \rightarrow 0$ for all $p\in [\sqrt{N},N)$. To see this, it is enough to consider the case $1<p <2$, otherwise we get the lower bound of $\lambda(\alpha)$ from propositions \ref{prop:LB1} and \ref{prop:LB3}. Note that if $1<p <2$, then the dimension $N$ can be either $2$ or $3$ (as $p^2\geq N$). If $N=3$, then the condition $p<\frac{N+1}{2}$ of Proposition \ref{prop:LB2} is automatically satisfied and therefore we get the lower bound of $\lambda(\alpha)$ from propositions \ref{prop:LB2} and \ref{prop:LB3}. Now, if $N=2$, then for $p \geq 2-\frac{1}{N}=\frac{3}{2}$  we obtain the lower bound of $\lambda(\alpha)$ from Proposition \ref{prop:LB1}, where as for $p<\frac{3}{2}=\frac{N+1}{2}$ we get the same from propositions \ref{prop:LB2} and \ref{prop:LB3}. }
\end{rem}


 \subsection{\bf Upper bounds}\label{sec_ub} 

 In this section we provide upper bounds of $\gl(\ga)$ as $\alpha \rightarrow 0$. In view of these upper bounds and the lower bounds obtained in the previous section, it follows that we have sharp two sided estimates for $\gl(\ga)$ as $\alpha \rightarrow 0$ for all $p<N\leq p^2$. 

\begin{prop} \label{prop: upperb-1}
Let $1< p< N < p^2$ and let $V\in  L^1(\R^N)\cap L^{\infty}(\R^N)$ be critical in $\R^N$  
satisfying \eqref{eq_Fuchs}. Suppose further that $W\in L^1(\R^N, \, \phi_0^p\dx)$ satisfies \eqref{condition}. Then
there exists a positive constant $K=K(N,p,V)$ such that 	
\begin{equation}
\limsup_{\alpha\to 0+}  \alpha^{-\frac{p(p-1)}{N-p}}\,  \lambda(\alpha) \leq -K\, \Big( \int_{\R^N} W \phi_0^p\dx\Big)^{\frac{p(p-1)}{N-p}}\, .
\end{equation}
\end{prop}

\begin{proof} 
Below we use the symbol $m(\alpha) \, \lesssim\,  M(\alpha)$ to indicate that there exists 
a constant $c>0$, independent of $\alpha$ and $W$, such that $m(\alpha)  \leq c\, M(\alpha) $ for all $\alpha >0$. 

\smallskip

\noindent To prove the desired estimate we will apply a test function argument. Let 
$$
f(x) = 
\begin{cases}
1  &\quad  \text{if}\quad  |x| \leq 1  \,,
				\\[4pt]	
e^{1-|x|} &\quad \text{if}\quad   1 < |x|\, ,
\end{cases}
$$
and define 
\begin{equation} \label{f-alpha}
f_{\alpha,t}(x) = f\big( t\, \alpha^{\frac{p-1}{N-p}}\, x\big)
\end{equation}
where $t>0$ is arbitrary. 
 Then, by monotone convergence,
\begin{equation} \label{potential-term}
\lim_{\alpha\to 0}   \int_{\R^N} W f^p_{\alpha,t} \, \phi_0^p\dx = \int_{\R^N} W \phi_0^p\dx =:\omega >0,
\end{equation}
for any $t>0$.
On the other hand, by \eqref{two-sided-pr} 
\begin{equation}  \label{Q0-upperb-1} 
Q_0[f_{\alpha, t}\, \phi_0] \lesssim 
\begin{cases}
\int_{\R^N} |\nabla f_{\alpha, t}|^p\,  \phi_0^p\dx  + \int_{\R^N} |\nabla f_{\alpha, t}|^2\, f_{\alpha,t}^{p-2}\,  \phi_0^2\, |\nabla \phi_0|^{p-2} \dx &
\mbox{ if }  p>2, \\[2mm]
 \int_{\R^N} |\nabla f_{\alpha, t}|^p\,  \phi_0^p\dx  &
\mbox{ if } p\leq 2.    
\end{cases}
\end{equation}

 Let 
$$
R_\alpha= t^{-1} \alpha^{\frac{p-1}{p-N}}\, .
$$
In view of  \eqref{eq_est} it follows that, as $\alpha\to 0$, 
\begin{equation} \label{calcul}
\begin{aligned}
\int_{\R^N}  f_{\alpha, t}^p\,  \phi_0^p\dx  & \, \asymp\,  \int_1^{R_\alpha} r^{\frac{p^2-N}{p-1}}\, \frac{\dr}{r} + \int_{R_\alpha}^\infty \exp\big(-t p\, \alpha^{\frac{p-1}{N-p}}\, r\big)\, r^{\frac{p^2-N}{p-1}}\, \frac{\dr}{r}
\\[4pt]
&\,  \asymp\, 
\alpha^{\frac{p^2-N}{p-N}}\, t^{\frac{N-p^2}{p-1}} + \alpha^{\frac{p^2-N}{p-N}}\, t^{\frac{N-p^2}{p-1}} \int_1^\infty e^{-s}\,  s^{\frac{p^2-N}{p-1}}\, \frac{\ds}{s} \\
&\,  \asymp\, 
\alpha^{\frac{p^2-N}{p-N}}\, t^{\frac{N-p^2}{p-1}}\, .
\end{aligned}
\end{equation}
Similarly, from \eqref{eq_est} and from the bound
\begin{equation}\label{phi0-grad}
|\nabla \phi_0(x)|\,  \lesssim\,  |x|^{\frac{p-N}{p-1}-1}  \qquad \text{as} \ \  |x|\to \infty,
\end{equation}
see \cite[Lem.~2.6]{fp},  we get
\begin{align*}
\int_{\R^N} |\nabla f_{\alpha, t}|^p\,  \phi_0^p\dx  & \,  \lesssim\, \alpha^{ \frac{p(p-1)}{N-p}}\, t^p \int_{\R^N}  f_{\alpha, t}^p\,  \phi_0^p\dx \\[4pt]
\int_{\R^N} |\nabla f_{\alpha, t}|^2\, f_{\alpha,t}^{p-2}\,  \phi_0^2\, |\nabla \phi_0|^{p-2} \dx  & \, \lesssim\,  \alpha\, t^{ \frac{(N-p)}{p-1}}\, .
\end{align*}
Hence 
\begin{align*}
\lambda(\alpha) & \leq \frac{Q_0[f_{\alpha, t}\, \phi_0] - \alpha  \int_{\R^N} W f^p_{\alpha,t} \, \phi_0^p\dx }{\int_{\R^N}  f_{\alpha, t}^p\,  \phi_0^p\dx }  \lesssim \,  \alpha^{ \frac{p(p-1)}{N-p}}\, \big(t^p - \omega\,  t^{\frac{p^2-N}{p-1}}\big)
\end{align*}
for all $\alpha>0$ and all $t>0$. Now the claim follows by optimizing in $t$.
\end{proof}

\begin{prop} \label{prop: upperb-2}
Let $N  = p^2$ and assume that $V$ and $W$ satisfy assumptions of Proposition \ref{prop: upperb-1}. Then 
there exists a positive constant $K=K(N,V)$ such that 	
\begin{equation}
\limsup_{\alpha\to 0+}  \frac{|\log \alpha|}{\alpha}\,   \lambda(\alpha) \leq -K \int_{\R^N} W \phi_0^p\dx \, .
\end{equation}
\end{prop}

\begin{proof} 
We follow the proof of Proposition \ref{prop: upperb-1}. Since $N=p^2$, using the family of functions defined \eqref{f-alpha} in combination with \eqref{eq_est} and \eqref{phi0-grad} we 
deduce from \eqref{Q0-upperb-1} that
\begin{equation} 
Q_0[f_{\alpha, t}\, \phi_0] \lesssim  \, \alpha\, t^p\, .
\end{equation}
with a constant which depends only on $V$ and $N$. Similarly, we get
\begin{align*}
\int_{\R^N}  f_{\alpha, t}^p\,  \phi_0^p\dx  & \, \asymp\,  \int_1^{R_\alpha} \, \frac{\dr}{r} + \int_{R_\alpha}^\infty \exp\big(-t p\, \alpha^{\frac{p-1}{N-p}}\, r\big)\,  \frac{\dr}{r}
  \asymp\, -\log t - \frac{p-1}{N-p}\, \log \alpha \, ,
\end{align*}
as $\alpha\to 0+$.
This together with \eqref{potential-term} implies that for any $t\in(0,1)$ and $\alpha$ small enough
\begin{align*}
\lambda(\alpha) & \leq \frac{Q_0[f_{\alpha, t}\, \phi_0] - \alpha  \int_{\R^N} W f^p_{\alpha,t} \, \phi_0^p\dx }{\int_{\R^N}  f_{\alpha, t}^p\,  \phi_0^p\dx } \, \lesssim \,  \frac{\alpha (t^p-\omega)}{-\log \alpha-\log t} , 
\end{align*}
Hence 
\begin{align*}
\limsup_{\alpha\to 0+}  \frac{|\log \alpha|}{\alpha}\,   \lambda(\alpha)& \lesssim t^p-\omega\, ,
\end{align*}
and the claim follows by letting $t \to 0+$.
\end{proof}

Finally, combining the lower bounds from Section \ref{sec_lb} and upper bounds from Section \ref{sec_ub}, we prove Theorem \ref{thm-superlinear}.
\begin{proof}[\bf Proof of Theorem \ref{thm-superlinear}]
The lower bound for the first assertion of the theorem (i.e., when 
$p<N<p^2$) is obtained from either Proposition~\ref{prop:LB1} or Proposition~\ref{prop:LB2}, depending on the specific case. The corresponding upper bound is provided by Proposition~\ref{prop: upperb-1}. To establish the second assertion of the theorem (i.e., when $N=p^2$), we apply Proposition~\ref{prop:LB3} for the lower bound and Proposition~\ref{prop: upperb-2} for the upper bound.
\end{proof}

 \begin{rem} 
 {\em In Proposition \ref{prop:LB-new} we present a variational proof of the lower bound in \eqref{N<p^2} which works without assuming that $V$ and $W$ are compactly supported. However, the latter works only if $2\leq p$.}
 \end{rem}
\begin{rem}
{\em Based on the estimates given in Theorem \ref{thm-superlinear}, it seems natural to expect that given a critical $V$ that decays fast enough, there exist $C_1(p,N,V)$ and $C_2(p,V)$ such that 
\begin{align*} 
\lim_{\alpha\to 0+} \alpha^{-\frac{p(p-1)}{N-p}}\, \lambda(\alpha) &= C_1(p,N,V) \, \Big( \int_{\R^N} W \phi_0^p\, \dx\Big)^{\frac{p(p-1)}{N-p}}\,  \qquad p<N < p^2\\[8pt]
\lim_{\alpha\to 0+} \frac{\log \alpha}{\alpha}\,    \lambda(\alpha)   &= C_2(p,V) \, \int_{\R^N} W \phi_0^p\, \dx  \qquad   \qquad \qquad\ \   N=p^2. 
\end{align*}
Establishing the existence of the limit and determining the values of the coefficients 
$C_1$ and $C_2$ remain open problems.}
\end{rem}

\appendix
\section{\bf Alternative proof of the lower bound for $p\geq 2$}
\noindent In the case $p\geq 2$ we have an alternative way to find an order sharp lower bound on $\lambda(\alpha)$.

\begin{prop} \label{prop:LB-new}
Let $2\leq p<N \leq p^2$. Suppose that $V\in L^\infty(\R^N)\cap L^1(\R^N)$  is critical in $\R^N$ and satisfies condition \eqref{eq_Fuchs}. Let  $W \in L^1(\R^N, \phi_0^p(x)\, \dx)$ satisfies \eqref{condition}. Then there exists a constant $C=C(p,N,V)>0$  such that 
\begin{equation} \label{lowerb-new}
\liminf_{\alpha \rightarrow 0_+} \alpha^{-\frac{p(p-1)}{N-p}}\lambda(\alpha)  \geq -C \left(\int_{\R^N} W(x) \, \phi_0^p(x) \dx\right)^{{\frac{p(p-1)}{N-p}}}  \,.
\end{equation}
 \end{prop}

\begin{proof}
In the proof below we denote by $c$ a generic positive constant whose value might change from line to line, and which may depend on $N,p$ and $V$ but 
not on $W$.

\smallskip

\noindent Let $\phi_0>0$ be the Agmon ground state of $Q_0$ normalized so that $\phi_0(0)=1$. 
Assume that $\phi_{\alpha} \in W^{1,p}(\R^N)$ is the minimizer of $\lambda(\alpha)$ with $\phi_{\alpha}>0$ and $\phi_{\alpha}(0)=1$. We write
$\phi_\alpha = f_\alpha\, \phi_0$ with $f_\alpha >0$. 
As in the proof of Theorem \ref{thm-linear} we conclude that $\phi_{\alpha} \rightarrow \phi_0$ in $L^{\infty}_{\loc}(\R^N)$. Hence
\begin{equation} \label{f-alpha-conv} 
f_\alpha \, \to \ 1 \qquad \text{in} \quad L^\infty_{\rm loc}(\R^N)\
\end{equation}
as $\alpha\to 0$. By \eqref{two-sided-pr}, 
$$
Q_0[\phi_\alpha] \ \geq\  c \int_{\R^N} \phi_0^p\, |\nabla f_\alpha|^p\dx \, .
$$
Let
\begin{equation} \label{m-alpha}
m_\alpha= \big( \min_{B_1} f_\alpha\big)^{-1}\
\end{equation} 
and note that $m_\alpha \geq  (f_\alpha(0))^{-1}= 1$. 
Now let $R>1$ be a number whose value will be specified later, and let  $\chi: \R^N\to \R$ be a cut-off function defined by 
$$
\chi(x) = m_\alpha \quad \text{if} \quad |x|\leq 1, \qquad \ \ \chi(x) = m_\alpha\, \Big( \frac{R-|x|}{R-1}\Big)_+  \quad \text{if} \quad |x|\geq 1\, .
$$ 
Since $\chi\leq m_\alpha$ and $|\nabla \chi|\leq  m_\alpha (R-1)^{-1}$, it follows that
\begin{equation} 
\| \phi_0  \nabla (\chi f_\alpha) \|_p^p\,  \leq\,  c\, m_\alpha^p  \| \phi_0  \nabla f_\alpha \|_p^p + c\, m_\alpha^p\, R^{-p}\, \| \phi_\alpha\|_p^p\, ,
\end{equation}
and hence 
$$
\| \phi_0 \nabla f_\alpha \|_p^p\, \geq\, c\, m_\alpha^{-p}  \| \phi_0  \nabla (\chi f_\alpha) \|_p^p - c \, R^{-p}\, \| \phi_\alpha\|_p^p\, .
$$
Altogether we obtain
\begin{align}  \label{lowerb-1-cap}
\lambda(\alpha) & = \frac{Q_0[\phi_\alpha] -\alpha \int_{\R^N} W\, \phi_\alpha^p\dx }{\| \phi_\alpha\|_p^p} \, \geq \, \frac{c\,  m_\alpha^{-p}   \| \phi_0  \nabla (\chi f_\alpha)\|_p^p  - \alpha \int_{\R^N} W\, \phi_\alpha^p\dx }{\| \phi_\alpha\|_p^p} -c\, R^{-p}\ .
\end{align} 
Let 
$$
\mathcal{F}_R = \big\{ u\in W^{1,p}(B_R\setminus B_1): u|_{\partial B_1} \geq 1 , \ u|_{\partial B_R}  =0\big\}\, .
$$
Since $\chi f_\alpha \geq 1$ for $|x|\leq 1$ and $\chi f_\alpha= 0$  for $|x|\geq R$, we can mimic the calculation of the capacity of the ball of radius one in the ball of radius $R$, see \cite[Sec.2.2.4]{m}. Using \eqref{eq_est} we estimate $ \| \phi_0  \nabla (\chi f_\alpha)\|_p^p$ as follows; 
\begin{align}  \label{lowerb-2-cap}
 \| \phi_0  \nabla (\chi f_\alpha)\|_p^p\  &\geq \ c  \inf_{u\in \mathcal{F}_R } \int_{B_R\setminus B_1} |\nabla u|^p\, |x|^{\frac{p(p-N)}{p-1}}\dx  
  \geq \ c\, |\s_N|   \inf_{u\in \mathcal{F}_R } \int_1^R  |\partial_\rho u|^p\, \rho^{d-1}\, d\rho\, , 
\end{align} 
where 
$$
d= \frac{p^2-N}{p-1}\, .
$$
A straightforward calculation  shows that the last integral in \eqref{lowerb-2-cap} attains its minimum at 
$$
u_0(\rho)= \frac{R^\nu-  \rho^\nu}{R^\nu-  1}\, \qquad \text{with} \quad \nu :=\frac{N-p}{(p-1)^2} >0.
$$
Inserting this into \eqref{lowerb-2-cap} gives
\begin{align} 
 \| \phi_0  \nabla (\chi f_\alpha)\|_p^p\  &  \geq \ c\,    \big( R^\nu-  1\big)^{1-p}
\end{align} 
with $c$ independent of $R$. Then, in view of \eqref{lowerb-1-cap},  
\begin{align}  
\lambda(\alpha)\  & \geq \frac{c\  m_\alpha^{-p} \big( R^\nu-  1\big)^{1-p}   - \alpha \int_{\R^N} W\, \phi_\alpha^p\dx }{\| \phi_\alpha\|_p^p} -c\, R^{-p}\, .
\end{align}
Now we chose $R=R_\alpha$ with $R_\alpha$ given by 
\begin{equation} 
c\  m_\alpha^{-p} \big( R_\alpha^\nu-  1\big)^{1-p} = \alpha \int_{\R^N} W\, \phi_\alpha^p\dx \, ,
\end{equation}
which implies
\begin{equation} \label{lowerb-3-cap}
\lambda(\alpha)\, \geq\,  - c\, R_\alpha^{-p}\, .
\end{equation} 
Since $m_\alpha\to 1$ by \eqref{f-alpha-conv}, and $ \int_{\R^N} W\, \phi_\alpha^p\dx \, \to  \int_{\R^N} W\, \phi_0^p\dx $, we have
$$
R_\alpha \, \geq\, c\,  \Big(\alpha  \int_{\R^N} W\, \phi_0^p\Big)^{\frac{p-1}{p-N}}
$$
with $c$ independent of $\alpha$, and the claim follows from \eqref{lowerb-3-cap}.
\end{proof}

\begin{center}
	{\bf Acknowledgments} 
\end{center}
U.D. is partially supported by the Basque Government through the BERC 2022-2025 program and by the Spanish Ministry of Science and Innovation: BCAM Severo Ochoa accreditation CEX2021-001142-S/MICIN/AEI/\allowbreak10.13039/501100011033 and CNS2023-143893. U.D. acknowledges that this project was initiated during his postdoc at Technion, Haifa, Israel, where he received
support from the Israel Science Foundation (grant 637/19) founded by the Israel Academy of Sciences and Humanities, and also from the Lady Davis Foundation. The authors also want to thank Marie-Fran\c{c}oise Bidaut-V\'{e}ron for the helpful reference \cite{LZ}.


\end{document}